\documentclass[12pt,reqno,twoside]{amsart}

\usepackage{amsmath}

\usepackage{amssymb}

\usepackage{epsfig}

\usepackage{color}

\numberwithin{equation}{section}

\newif\ifdraft\drafttrue

\draftfalse

\font\sb = cmbx8 scaled \magstep0

\font\sn = cmssi8 scaled \magstep0
\font\si = cmti8 scaled \magstep0

\long\def\commargin#1{\ifdraft{\marginpar{\si #1}}\else\ignorespaces\fi}

\long\def\commarginnew#1{\ifdraft{\marginpar{\sb #1}}\else\ignorespaces\fi}

\newcommand\A{
{a}}
\newcommand\B{
{b}}

\newcommand\mr{M_{m,n}}

\newcommand\amr{$Y\in M_{m,n}
$}

\newcommand\ba{badly approximable}

\newcommand\da{Diophantine approximation}
\newcommand\di{Diophantine}

\newcommand\hs{homogeneous space}

\newcommand\stl{strongly tree-like}

\newcommand\ssm{\smallsetminus}

\newcommand\name[1]{\label{#1}{\ifdraft{\sn [#1]}\else\ignorespaces\fi}}
\newcommand\bname[1]{{\ifdraft{\sn [#1]}\else\ignorespaces\fi}}

\newcommand\eq[2]{{\ifdraft{\ \tt [#1]}\else\ignorespaces\fi}\begin{equation}\label{eq:#1}{#2}\end{equation}}

\newcommand {\equ}[1]     {\eqref{eq:#1}}
\newcommand{\under}[2]{\underset{\text{#1}}{#2}}

\newcommand{\goth}[1]{{\mathfrak{#1}}}

\newcommand{\R}{{\mathbb{R}}}

\newcommand{\Z}{{\mathbb{Z}}}

\newcommand{\C}{{\mathbb{C}}}

\newcommand{\E}{{\mathbf{e}}}

\newcommand{\N}{{\mathbb{N}}}

\newcommand{\vs}{{\bf{s}}}

\newcommand{\Int}{\operatorname{Int}}
\newcommand{\Tr}{\operatorname{Tr}}

\newcommand{\SL}{\operatorname{SL}}

\newcommand{\ggm}{G/\Gamma}

\newcommand{\diag}{{\rm diag}}

\newcommand{\End}{{\rm End}}

\newcommand {\ignore}[1]  {}

\newcommand{\spa}{{\rm span}}

\newcommand{\diam}{{\rm diam}}

\newcommand{\df}{{\, \stackrel{\mathrm{def}}{=}\, }}

\newcommand{\FF}{{\mathcal{F}}}

\newcommand{\x}{{\mathbf{x}}}

\newcommand{\vy}{{\bf y}}

\newcommand{\vr}{{\bf r}}

\newcommand{\vn}{{\bf n}}

\newcommand{\p}{{\bf p}}

\newcommand{\vp}{{\bf p}}

\newcommand{\vq}{{\bf q}}

\newcommand{\til}{\widetilde}

\newcommand{\supp}{{\rm supp}}

\newcommand{\BA}{{\bold{Bad}}}

\newcommand{\sm}{\smallsetminus}

\newcommand{\vre}{\varepsilon}

\newcommand\hd{Hausdorff dimension}

\newcommand\nz{\smallsetminus \{0\}}
\newcommand{\ay}{{\bf{A}_\infty}}

\newcommand{\ca}{{\mathcal A}}
\newcommand{\cd}{{\mathcal D}}

\newtheorem{thm}{Theorem}[section]

\newtheorem{lem}[thm]{Lemma}

\newtheorem{prop}[thm]{Proposition}

\newtheorem{cor}[thm]{Corollary}

\title[Modified Schmidt games]{Modified Schmidt games 
 and Diophantine approximation with weights}

\author{Dmitry Kleinbock}

\address{Brandeis University, Waltham MA
02454-9110 {\tt kleinboc@brandeis.edu}}

\author{Barak Weiss}

\address{Ben Gurion University, Be'er Sheva, Israel 84105
{\tt barakw@math.bgu.ac.il}}

\date{
August 2009}

\begin{document}

 \begin{abstract} We 
show
 that the sets of weighted badly approximable vectors in $\R^n$ are winning sets of certain games, which are modifications of $(\alpha,\beta)$-games introduced by W.\ Schmidt in 1966. The latter winning property  is stable with respect to countable intersections, and is shown to imply full \hd. 

\end{abstract}
\maketitle

\commargin{removed `simple argument'}
\section{Introduction} \name{intro} A classical result of Dirichlet \commargin{The $\to$ A}
states that for any $\x \in \R^n$  there are infinitely many $q\in\N$
such that $\|q \x - \p \| < 
{q^{-1/n}}$ for some $\p\in\Z^n$. 
One says that $\x \in \R^n$ is {\sl badly approximable\/} if the right hand side of the 
above inequality cannot be improved by an arbitrary positive constant. In other words,
if there is
$c>0$ such that for any $\p \in \Z^n, \, q \in \N$ one has
\begin{equation}
\label{eq: defn ba}
\|q \x - \p \| \geq \frac{c}{q^{1/n}}\,.
\end{equation}
Here $\|\cdot\|$ can be any norm on $\R^n$, which unless otherwise
specified
will be chosen to be the supremum norm. 
We denote the set of all badly approximable vectors in $\R^n$ by
$\BA_n$, or $\BA$ if the dimension is clear from the context.
\commargin{Introduced thickness as in the big paper, to
get rid of dimension at any point}  It is
well known that Lebesgue measure of $\BA$ is zero; but 
nevertheless this set is quite large. Namely it is {\sl thick\/}, that is,
 its intersection with every open set in $\R^n$ has full \hd\
(Jarnik \cite{Jarnik} for $n = 1$, 
\commargin{added reference to \cite{Schmidt:book}} Schmidt
\cite{Schmidt games, Schmidt:book} for $n > 1$). 
 In fact Schmidt established a stronger property of the set $\BA$: that it is a 
 so-called winning set
 for a certain game which he invented for that occasion, see \S\ref{sec: games} for 
 more detail. In particular, the latter
 property implies that for any countable sequence of 
 similitudes (compositions of translations and
homotheties)
${f}_i:\R^n\to\R^n$, the intersection $\cap_i {f}_i(\BA)$
is thick as well.  

Our purpose in this paper is to introduce a modification
of Scmidt's game, and apply it to 
similarly study a weighted generalization of the notion of \ba\
vectors. Take a vector $\vr = (r_i \mid 1\le i\le n)$ such that 
\eq{defn r}{
r_i > 0\quad\text{and}\quad\sum_{i=1}^m r_i = 1\,,
}
thinking of each $r_i$ as of a weight assigned to $x_i$. It is easy to show that the following
multiparameter version of the aforementioned Dirichlet's result holds: 
 for $\vr$ as above and any $\x  = (x_1,\dots,x_n)\in \R^n$  there are infinitely
many $q\in\N$ such that \eq{rba}{
\max_{1\le i \le n}|q x_i- p_i|^{1/r_i} < 
{q^{-1}} \text{ for some }\p  = (p_1,\dots,p_n)\in\Z^n\,.}
 This motivates the following
definition: say that $\x$ is {\sl $\vr$-badly approximable\/} if the right hand side of \equ{rba} cannot be improved by an arbitrary positive constant; in other words,
if there is
$c>0$ such that for any $\p \in \Z^n, \, q \in \N$ one has
\eq{defn rba}
{\max_{1\le i \le n}|q x_i- p_i|^{1/r_i} \geq \frac{c}{q}\,.
}
Following \cite{PV-bad} and \cite{KTV},  denote by $\BA(\vr)$ the
set of $\vr$-badly approximable vectors. 
It is not hard to make sense of the above definition when one or more of the components of 
$\vr$ are equal to zero: one simply needs to ignore these components following a convention
 $a^{\infty} = 0$ when $0 \le a < 1$. For example,  $\BA(1,0) = \BA_1
\times \R$ and  $\BA(0,1) = \R\times \BA_1$. Also it is clear that
$\BA_n  = \BA(\vn)$ where \eq{def n}{\vn = (1/n,\dots,1/n)\,.} 
 
One of the main results of 
\cite{PV-bad} states that the set
$\BA(\vr)$
is thick for any $\vr$ as above (this was conjectured
earlier in \cite{K-india}). A complete proof is given in \cite{PV-bad} for the
case $n = 2$, but the method, based on 
some ideas of Davenport, straightforwardly extends
to higher dimensions as noted by the authors of  \cite{PV-bad}. A slightly different
 proof can be found in  \cite{KTV}. In this paper we present 
a modification (in our opinion, a simplification)
of the argument from the aforementioned papers which yields a stronger result. 
\commarginnew{added more intro material}
Namely, in \S\S \ref{sec: games}--\ref{sec: contr} we describe a variation of Schmidt's game, 
which we call {\sl modified Schmidt game\/} (to be abbreviated by MSG)
induced by a family of contracting automorphisms of $\R^n$, and study properties of winning
sets of those modified games. We show that   winning sets of MSGs are thick
(Corollary \ref{cor: msg3}),
 and a countable
intersection of sets winning for the same game is winning as well
(Theorem \ref{thm: countable general}). In \S\ref{sec: win}
we prove

\begin{thm}\name{thm: main}
Let  $\vr$ be as in \equ{defn r}, and let $\mathcal{{F}}^{(\vr)} =
\{{\Phi}^{(\vr)}_t: t > 0\}$ be the one-parameter 
semigroup of linear contractions of $\R^n$ defined by 
\eq{defn ar}{{\Phi}^{(\vr)}_t = \diag(e^{-(1+r_1)t},\dots,e^{-(1+r_n)t)})\,.}
Then the set $\,\BA(\vr)$
is a winning set for the modified Schmidt game (to be abbreviated by MSG)
induced by $\mathcal{{F}}^{(\vr)}$; in particular, it is thick.
\end{thm}


Note that the original
Schmidt's game can be viewed as a MSG induced by \commarginnew{rearranged this part}
the family of homotheties of $\R^n$; thus Schmidt's theorem on 
$\BA$ being a winning set is a special case of Theorem \ref{thm: main}.
The countable intersection property of winning sets of MSGs  makes it possible to 
intersect $\BA(\vr)$ with its countably many
dilates and 
translates  (see a remark after Theorem \ref{thm: precise}),
as well as establish,
in a simpler way, another result of \cite{PV-bad},
namely that the set \eq{tripleint}{\BA(r_1,r_2) \cap \BA(1,0)  \cap \BA(0,1)} is thick 
for any 
$0 < r_1,r_2 < 1$ with $r_1 + r_2 = 1$. This and other  concluding remarks are made in \S\ref{sec: next}.

\smallskip

{\bf Acknowledgements:}
The authors are grateful to the hospitality of Tata Institute of Fundamental Research
\commargin{removed info about Dani's conference, kind of irrelevant}
(Mumbai)  where they 
had several conversations which eventually led to results described in this paper.
Thanks are also due \commarginnew{more thanks added}
 to Elon Lindenstrauss for motivating discussions, to the referee
for useful comments, 
 and to Max Planck Institute for Mathematics
(Bonn) where the paper was completed.
This work was supported by BSF grant 2000247, ISF grant 584/04, and NSF
Grants DMS-0239463, DMS-0801064.


\section{Modified Schmidt Games}\name{sec: games} \subsection{Schmidt's game}\name{special}
Let $(E,d)$  be a complete metric space,
and let $\Omega \df E\times \R_+$ (the set of formal balls in $E$). Following  
\cite{Schmidt games}, define a 
partial ordering (Schmidt's containment) on $\Omega$ as follows:					
\eq{cont} 									
{(x',r') \le_s (x,r)\quad \iff \quad d(x',x) + r' \leq r\,. 							
}
To each 
 pair $(x,r)\in\Omega$ we associate a closed ball in $E$ via the `ball' function $B$:
\commargin{changed notation so that $B(\cdot)$ stands for closed balls now}
$ B(x,r) \df \{y\in E : d(x,y) \le r\}$. 
Note that $(x', r') \leq_s (x,r)$ implies $B(x', r') \subset B(x,r)$;
while in Euclidean space these conditions are in fact equivalent, in a
general metric space the converse need not hold.


Now pick $0 < \alpha,\beta < 1$  and consider  
the following game, commonly referred to as {\sl Schmidt's game\/}, played by two players, whom we will call\footnote{Schmidt originally named his players `white' and  `black'; in the subsequent literature
letters $A$ and $B$ were often used instead. We are grateful to Andrei Zelevinsky for suggesting the
Alice/Bob  nomenclature following a convention common in computer
science.}
Alice and Bob. \commarginnew{Renamed the players and added a footnote}
The game starts with Bob  choosing $x_1\in E$ and $r > 0$, 
hence specifying a pair $\omega_1 \df (x_1,r)$. 
Alice  may now choose any point $x_1'\in E$ provided that 
$\omega_1'\df (x_1',\alpha r)\le_s\omega_1$. 
Next,  Bob chooses a point $x_2\in E$ such that
$\omega_2\df(x_2,\alpha \beta r)\le_s \omega_1'$, and so on. 
Continuing in the same manner, one obtains a nested sequence of balls in $E$:
$$
B(\omega_1) \supset B(\omega_1') \supset B(\omega_2) \supset B(\omega_2') \supset \ldots\supset 
 B(\omega_k) \supset  B(\omega_k') \supset \ldots
$$
A subset $S$ of $E$ is called {\sl $(\alpha,\beta)$-winning \/} if Alice
can play in such a way that the unique point of intersection 
\eq{int}{
\bigcap_{k = 1}^\infty  B(\omega_k) = \bigcap_{k = 1}^\infty  B(\omega_k')}
lies in $S$, no matter how Bob plays. $S$ is called  
{\sl $\alpha$-winning \/} if it is   $(\alpha,\beta)$-winning for all $\beta > 0$,
and {\sl winning \/} if it is   $\alpha$-winning for some $\alpha > 0$.
We will denote balls chosen by Bob (resp., Alice) by 
$B_k \df  B(\omega_k) $ and $A_k \df  B(\omega_k')$. 
 \smallskip

The following three theorems are due to Schmidt \cite{Schmidt
games}. 

\begin{thm}
\name{thm: countable} Let $S_i\subset E$, $i\in\N$, be
a sequence of $\alpha$-winning sets for some $0 < \alpha < 1$; 
then $\cap_{i = 1}^\infty S_i$
is also  $\alpha$-winning.
\end{thm}

\begin{thm}
\name{thm: full dim} Suppose the game is played on $E=\R^n$ with the
Euclidean metric; then any winning set is thick.
\ignore{
\eq{conclusion full dim}{
\dim(S\cap U) \ge \frac{\log c_n\beta^{-n}}{|\log \alpha\beta|}\,,
}
where $c_n$ is a constant depending only on $n$; in particular any $\alpha$-winning 
subset of $\R^n$  has \hd\ $n$.}
\end{thm}

\begin{thm}
\name{thm: ba} For any $n\in\N$, $\BA_n$ is $(\alpha,\beta)$-winning
whenever $2\alpha < 1 + \alpha\beta$; 
in particular, it is  $\alpha$-winning for any $0 < \alpha \le 1/2$.
\end{thm}

It can also be shown that for various classes of continuous maps of metric spaces,
\commargin{Removed details}
the images of winning sets are also winning for suitably modified values of constants.
See \cite[Theorem 1]{Schmidt games} and 
\cite[Proposition 5.3]{Dani survey} for details.



\subsection{A modification}\name{modified}
We now introduce a variant of this game, which is in fact a special case 
of the general framework of $(\frak F, \frak S)$-games described by
Schmidt in \cite{Schmidt games}.  As before, let $E$  be a complete metric space,
and let $\mathcal{C}(E)$ stand for the set of nonempty compact subsets of $E$.
Fix $t_*\in\R\cup\{-\infty\}$  and
define $\Omega = E \times (t_*,\infty)$\footnote{Note that everywhere one could replace $\R$ with
some fully ordered semigroup. This more general setup presents no
additional difficulties but we omit it to simplify notation.}. Suppose in addition that we are given
\begin{itemize}
\item[(a)] a partial ordering $\le$ on $\Omega$, and
\item[(b)] a monotonic function $\psi: (\Omega,\le)\to \big(\mathcal{C}(E),\subset\big)$.
\end{itemize}
Here monotonicity means that
$\omega' \le\omega$ implies 
$\psi(\omega') \subset \psi(\omega)$.
Now fix $a_* \geq 0$ and suppose that
the following property holds:
\begin{itemize}
\item[(MSG0)] 
For any $(x,t)\in\Omega$ and any $s >a_*$   there exists $x'
\in E$ such that $(x', t+s)\le (x,t)$.
\end{itemize} 

Pick two numbers $\A$ and $\B$, both bigger than $a_*$.  
Now Bob begins the $\psi$-$(a, b)$-game by choosing  $x_1\in E$ and $t_1 > t_*$, 
hence specifying a pair $\omega_1 \df (x_1,t_1)$. 
Alice  may now choose any point $x_1'\in E$ provided that 
$\omega_1'\df (x_1',t_1 + \A)\le\omega_1$. 
Next,  Bob chooses a point $x_2\in E$ such that
$\omega_2\df(x_2,t_1 + \A + \B)\le \omega_1'$, and so on. 
Continuing in the same manner, one obtains a nested sequence of compact subsets of $E$:
$$
B_1 =  \psi(\omega_1) \supset A_1 =  \psi(\omega_1') \supset 
\ldots\supset 
B_k = \psi(\omega_k) \supset A_k = \psi(\omega_k') \supset \ldots
$$
where $\omega_k =  (x_k,t_k)$ and  $\omega_k' =  (x_k',t_k')$ with \eq{si ti}{t_k =   t_1+(k-1)(\A+\B)\text{ and  }
t_k' =   t_1+(k-1)(\A+\B) + \A 
\,.}
Note that Bob
and Alice can always make their choices by virtue of (MSG0), and that the intersection
 \eq{intmodified}{
\bigcap_{k = 1}^\infty \psi(\omega_k) = \bigcap_{k = 1}^\infty \psi(\omega_k')}
 is  nonempty and compact. 
Let us say that
$S\subset E$ is {\sl $(a, b)$-winning\/}
for the {\sl modified Schmidt game corresponding to\/} $\psi$, to be
abbreviated as $\psi$-MSG, 
if 
Alice
can proceed in such a way that the set \equ{intmodified} is contained in $S$  no matter how 
Bob plays. 
Similarly, say
 that $S$ is an  {\sl $a$-winning\/} set of the game if 
 $S$ is $(a, b)$-winning for any choice of $b > a_*$, and that $S$ is
{\sl winning} if it is $a$-winning for some $a > a_*$. 
Note that we are
suppressing $a_*$ and $t_*$ from our notation, hopefully this will cause
no confusion. 


Clearly the game described above coincides with the original $(\alpha,
\beta)$-game if we let 
\eq{rn}{
\begin{aligned}
\psi(x,t) = B(x,e^{-t}), \quad (x',t')\le(x,t) \Leftrightarrow (x',e^{-t'}) \le_s (x,e^{-t}),\\
\A =- \log \alpha, \ \B =- \log \beta, \ a_*=0,\ t_* = - \infty\,.\qquad\quad
\end{aligned}}

Here is some more notation which will be convenient later. For $t > t_*$ we let
$$\Omega_t\df \{(x,t) : x\in E\}\,,$$ so that $\Omega$
is a disjoint union of the `slices' $\Omega_t$, $t > t_*$. Then for $s > 0$ 
and $\omega\in\Omega_t$ define
$$
I_s(\omega) \df \{\omega'\in \Omega_{t+s} : \omega' \le \omega\}\,.
$$
In other words, $
I_\A(\omega)$ and $
I_\B(\omega)$ are the sets of allowed moves of Alice and Bob
respectively starting from position $\omega$. Using this notation condition (MSG0) can be
reworded as 
\begin{itemize}
\item[(MSG0)] $I_s(\omega) \ne\varnothing$ for any   $\omega\in\Omega$, $s >a_*$.
\end{itemize}

\subsection{General properties}\name{general}
Remarkably, even in the quite general setup described in \S\ref{modified}, 
an analogue of Theorem \ref{thm: countable} holds and can be 
proved by a verbatim repetition of the argument 
from \cite{Schmidt games}:


\begin{thm}
\name{thm: countable general}  Let a metric space $E$, 
partially ordered $\Omega= X\times (t_*,\infty)$ and $\psi$ be as above, 
let $\A > a_*$,
and let $S_i\subset E$, $i\in\N$, be 
a sequence of $\A$-winning sets of the  $\psi$-MSG. Then $\cap_{i = 1}^\infty S_i$
is also  $\A$-winning.
\end{thm}

\begin{proof} Take an arbitrary $\B > a_*$, and make Alice play according to the
 following
rule. At the first, third, fifth \dots\ move Alice will make a choice according
to an $(\A, 2\A + \B,S_1)$-strategy (that is, will act as if playing an 
$(\A, 2\A + \B)$-game trying to reach 
$S_1$). At the second, sixth, tenth \dots\ 
move she will use an $(\A, 4\A + 3\B,S_2)$-strategy. In general, at the $k$th move,
where $k \equiv 2^{i -1} (\text{mod}\, 2^i)$, she 
will play the
$\big(\A, \A + (2^i - 1)(\A + \B)\big)$-game trying to reach a point in $S_i$.
It is easy to see that, playing this way, Alice can enforce that the intersection
of the chosen sets belongs to $S_i$ for each $i$.\end{proof}

Here are two more general observations about MSGs
and their winning sets.

\begin{lem}
\name{lem: dummy}   Let $E$, $\Omega$ and $\psi$ be as above, and suppose
 that $S\subset E$,  $\A, \B
> a_*$ and ${t_0} > t_*$ are such that 
whenever Bob initially chooses $\omega_1 \in \Omega_{t}$ with $t \ge {t_0}$, 
Alice can win the game. 
Then $S$ is an $(a,b)$-winning set of the $\psi$-MSG.
\end{lem}


\begin{proof} Regardless of the initial move of Bob, Alice can 
make arbitrary (dummy) moves waiting until $t_k$ becomes at least ${t_0}$,
and then apply the strategy  he/she is assumed to have. 
\end{proof}

This lemma shows that  the collection of  $(a,b)$-winning sets of a given $\psi$-MSG
depends only on 
the `tail' of the family $\{\Omega_t\}$ 
and not on the value of $t_*$. 
\commargin{removed  `ignore $t_*$', in fact we have $t_* = \infty$}

\begin{lem}
\name{lem: product}
Let $E_1,E_2$ be complete metric spaces, and
consider two games
corresponding to
%
$\psi_i: \Omega_i\to\mathcal{C}(E_i)$, where 
$\Omega_i = E_i \times (t_*,\infty)$.  Suppose that $S_i\subset E_i$ is an
$(\A, \B)$-winning set of the $\psi_i$-MSG, $i = 1,2$. Then
$S_1\times S_2$ is an $(\A, \B)$-winning set of the $\psi$-MSG played
on $E = E_1\times E_2$ with the product metric, where $\psi$
is defined by $$\psi(x_1,x_2,t) = \psi_1(x_1,t)\times \psi_2(x_2,t)\,.$$
\end{lem}

\begin{proof} Play a game in the product space by playing two separate games in each 
of the factors.
\end{proof}

\commargin{Removed details here as well}
It is also possible to write down conditions on $f:E\to E$, quite restrictive in general,
 sending winning sets of
the $\psi$-MSG to winning sets. We will exploit this theme in
\S \ref{images contr}.

\subsection{Dimension estimates}\name{dimest}
Our next goal is to generalize Schmidt's lower estimate for the
  \hd\  of winning sets in $\R^n$. Note that in general it is not true,
 even for original  Schmidt's game  \equ{rn} played on an arbitrary
complete metric space, that winning sets have positive \hd: see
Proposition \ref{prop: lower dim example} for a counterexample. 
We are going to make some assumptions 
that will be sufficient to
ensure that a winning set for the $\psi$-MSG is big enough. Namely we will assume:

\begin{itemize}
\item[(MSG1)]  For any open $\varnothing\ne U \subset E$ there is 
$\omega \in
\Omega$ such that $\psi(\omega) \subset U$. 
\item[(MSG2)] There exist $C,\sigma > 0$ such that {\rm diam}$\big(\psi(\omega)\big) \le C e^{-\sigma t}$ 
for  all $t \ge t_*$, $\omega \in
\Omega_t$.
\end{itemize}

We remark that it follows from (MSG1) that any $(\A,\B)$-winning set of the game is dense,
and from (MSG2) that the intersection \equ{intmodified} consists of a single point.

To formulate two additional assumptions, we suppose that we are given 
a locally finite Borel measure $\mu$ on $E$ satisfying the following conditions:
\begin{itemize}
\item[($\mu$1)] 
 $\mu\big(\psi(\omega)\big) > 0$ for any $\omega \in
\Omega$.  
\item[($\mu$2)] 
\commargin{added $\rho$}
For any $\A > a_*$ there exist $c,\rho>0$ 
with the following property:
$\forall\,\omega \in \Omega$ with $\diam\big(\psi(\omega)\big) \le \rho$ 
and  $\forall\,\B > a_*$  $\exists\,\theta_1,\dots,\theta_N\in I_\B(\omega)$ 
such that $\psi(\theta_i), \ i = 1,\dots,N,$ are essentially disjoint, and that
for every  $\theta_i'\in I_\A(\theta_i)$,  $ i = 1,\dots,N$,
 one has
$$\mu\big(\bigcup_i \psi(\theta_i')\big) \geq c \mu\big(\psi(\omega)\big)\,.$$
\end{itemize}

The utility of the latter
admittedly cumbersome condition will become clear in the sequel, see
Proposition \ref{cor: Federer}. 
Here and hereafter we say that $A,B\subset E$ are {\sl essentially disjoint\/} if $\mu(A\cap B) = 0$. 
In particular, it follows from ($\mu$1) and (MSG1) that such a measure $\mu$ must have 
full support (this will be our standing assumption from now on). Also, note that (MSG0)
is a consequence of  
($\mu$2).


Now recall that the {\sl lower pointwise dimension\/} of $\mu$ 
at $x\in E$ is defined by\footnote{This and other properties,
such as the Federer property  introduced in \S\ref{sec: fulldim}, are
usually stated for open balls, but versions with closed balls
are clearly equivalent, modulo a slight change of constants if necessary.}
$$
\underline{d}_\mu(x) \df \liminf_{r\to 0}\frac{\log \mu\big(B(x,r)\big)}{\log r}\,,
$$
and for $U\subset E$ let us put $$
\underline{d}_\mu(U) \df \inf_{x\in U}\underline{d}_\mu(x) \,.
$$
It is known, see e.g.\ \cite[Proposition 4.9(a)]{Falconer} or \cite[Theorem 7.1(a)]{Pesin}, that $
\underline{d}_\mu(U)$ is a lower bound for the \hd\ of  $U$ for any
nonempty open $U\subset E$, and very often it is possible to 
choose $\mu$ such that $
\underline{d}_\mu(x)$ is equal to $\dim(E)$ for every $x$.
For instance this is the case when $\mu$ {\sl satisfies a power law\/}, 
that is, if there exists $\gamma, c_1, c_2, r_0 > 0$ such that
\eq{pl}{c_1r^\gamma \le \mu\big(B(x,r)\big) \le c_2 r^\gamma\text{  whenever }r \le r_0\text{ and }
x\in E
}
(then necessarily $\dim(U) = \gamma$ for any nonempty open $U\subset E$).

\ignore{
Now for any $s > 0$  define $N(s)$ to be the maximal $N$
such that for any $t \ge t_*$ and any $x\in E$ there exist $N$ essentially disjoint 
sets $D_{ x_1,t+s}  ,\dots D_{x_N,t+s}  $ of $D_{t+,s}$ 
which are contained in
$D_{x,t}$. Here `essentially disjoint' means zero measure of the intersection.
Note that in order for $(a,b)$-winning sets to exist we must at least demand that $N(a) > 0$,
otherwise  Alice might be left with no  moves to make. 
}

\begin{thm}
\name{thm: WD} 
Suppose that   $E$, $\Omega$, $\psi$ 
and a  measure $\mu$ on $E$
are such that {\rm (MSG0--2)} and  {\rm ($\mu$1--2)}  hold.
Take $\A,\B > a_*$ and let \commargin{added $(\A,\B)$-winning}
$S$ be an $(\A,\B)$-winning set  of the $\psi$-MSG. Then for any 
open $\varnothing \ne U\subset E$, one has \eq{wd}{\dim (S \cap
U) \geq \underline{d}_{\mu}(U)  +
\frac1\sigma\left({ \frac{\log c}{\A+\B}  }\right) \,,} where $\sigma$ is as in {\rm (MSG2)}
and $c$ as in  {\rm ($\mu$2)}. In particular, $\dim (S \cap
U) $ is not less than $ \underline{d}_{\mu}(U) $ whenever $S$ is winning.
\end{thm}

Before proving this theorem let us observe that it generalizes Theorem \ref{thm: full dim}, 
with Lebesgue measure playing the role of $\mu$. Indeed, conditions (MSG0--2) are trivially satisfied in the case \equ{rn}. 
It is also clear that  ($\mu$1) holds and that  $ \underline{d}_{\mu}(x) = n$ for all $x\in\R^n$. 
As for ($\mu$2), note that there exists a constant $\bar c$,
depending only on $n$, such that for any  $0<\beta < 1$, the unit ball in $\R^n$ contains
a disjoint collection of closed balls $D_i'$ of radius $\beta$ of relative measure at least $\bar c$; and no matter how 
balls $D_i\subset D_i'$ of radius $\alpha\beta$ are chosen, their total relative measure will
not be less than $\bar c \alpha^n$. 
Rescaling, one obtains ($\mu$2). 
See  Lemma \ref{prop: msg3} and Proposition \ref{cor: Federer} for  further generalizations.

\smallskip

For the proof of Theorem \ref{thm: WD} 
we will use a construction suggested in \cite{McMullen,
Urbanski} and formalized in  \cite{KM}. 
Let $E$ be a  complete  metric space  equipped with a locally finite Borel measure $\mu$.  
Say that a
countable family $\ca$ of compact subsets of $E$ of positive measure
        is {\sl
tree-like\/} (or {\sl tree-like with respect to $\mu$}) if $\ca$ is
the 
union of finite subcollections 
$\ca_k$, $k\in\Z_+$, such that $\ca_0 = \{A_0\}$ and the following four
conditions are satisfied:
\begin{itemize}
\item[(TL0)]
$\mu(A)  > 0$ for any $A\in \ca\,;$
\smallskip
\item[(TL1)]
$\forall\,k\in\N\quad \forall\, A, B \in \ca_k\quad\text{either
}A=B\quad\text{or}\quad {\mu}(A\cap B) = 0\,;$
\smallskip
\item[(TL2)]
$\forall\,k\in\N\quad \forall\,B \in \ca_k\quad\,\exists\,A\in
\ca_{k-1} \quad\text{such that}\quad B\subset A\,;$
\smallskip
\item[(TL3)]
$\forall\,k\in\N\quad \forall\,A \in \ca_{k-1} \quad\, 
\exists\,B\in
\ca_{k} \quad\text{such that}\quad B\subset A
$.

\end{itemize}

Then one has $A_0\supset \cup \ca_1\supset \cup \ca_2\dots$,
a decreasing intersection of nonempty compact sets
(here and \commarginnew{The referee complained about this notation, but now it is explained}
elsewhere we denote $\cup \mathcal{A}_k = \bigcup_{A\in \mathcal{A}_k}A$), which defines
the (nonempty) {\sl limit set\/} of $\ca$, 
$$
\ay = \bigcap_{k\in
\N}\cup \ca_k\,.
$$
Let us also define the {\sl $k$th
stage diameter\/} $d_k(\ca)$ of $\ca$:
$$
d_k(\ca)\df\max_{A\in\ca_k}\text{diam}(A)\,,
$$
and say that $\ca$ is {\sl strongly tree-like\/} if it is
tree-like and in addition
\begin{itemize}
\item[(STL)] \qquad
$\lim_{k\to\infty}d_k(\ca) = 0\,.$
\end{itemize}

Finally, for $k\in \Z_+$ 
let us define the $k$th stage `density of children' of  $\ca$
by
$$
\Delta_k(\ca) \df \min_{B\in\ca_k}\frac{{\mu}\big(\cup  \ca_{k+1}\cap B\big)}{{\mu}(B)}\,,
$$
the latter being always positive due to (TL3).
The following lemma, proved in \cite{bad} and  generalizing results of C.~McMullen 
\cite[Proposition 2.2]{McMullen}
and  M.~Urbanski \cite[Lemma 2.1]{Urbanski},  provides a needed lower estimate
for the \hd\ of $\ay$:
\begin{lem}
\name{lem: Urbanski}
Let $\ca$ be a strongly
tree-like (relative to
$\mu$) collection  of subsets of $A_0$.
Then
for any open $U$ intersecting $\ay$ one has
$$
\dim (\ay \cap U) \geq \underline{d}_\mu(U)
 -
\limsup_{k\to\infty}\frac{ \sum_{i=0}^{k}\log\Delta_i(\ca)}{\log\, d_{k}(\ca)}\,.
$$
\end{lem}
Note that even though \cite[Lemma 2.5]{bad} is stated for $E = \R^n$, its proof,
including  the Mass Distribution Principle on which the lower estimate for the \hd\ is based,
is valid in the generality of an arbitrary complete metric space. 

\begin{proof}[Proof of Theorem \ref{thm: WD}] Our goal is to find a 
\stl\ collection $\ca$ of sets whose limit set is a subset of $S \cap
U$. It will be constructed by considering 
possible moves for  Bob at each stage of the game, and the
corresponding counter-moves specified by Alice's winning
strategy. 
Fix $\A,\B > a_*$ for which $S$ is $(\A,\B)$-winning. By assumption (MSG1), 
Bob may begin the game by choosing $t_1
> t_*$ and $\omega_1 \in \Omega_{t_1}$ such that $\psi(\omega_1)\subset U$
and $\diam\big(\psi(\omega_1)\big) < \rho$, where $\rho$ is as in ($\mu$2). 
Since $S$ is winning, Alice can choose $\omega_1' \in I_\A(\omega)$ 
such that  $A_0\df \psi(\omega_1')$  has nonempty intersection with
$S$;  it will be the ground set of our tree-like 
family. 

Now let  $\theta_1, \ldots, \theta_N\in  I_\B(\omega_1')$ be as in ($\mu$2) for
$\omega = \omega_1'$. Each of
\commarginnew{added some more explanations, is this enough? feel free to edit...}
these could be chosen by Bob at the next step of the game. Since $S$ is $(\A,\B)$-winning,
for each  of the above choices $\theta_i$  Alice can pick  $\theta_i' \in  I_\A(\theta_i)$ 
such that every  
sequence of possible further moves of  Bob can be counter-acted by Alice resulting
in her victory in the game. 
The collection of images $ \psi(\theta_i')$ of these choices of Alice, essentially disjoint 
in view of ($\mu$2),  will comprise the first
level $\ca_1$ 
of the tree. Repeating the same for each of the choices we obtain $\ca_2$, $\ca_3$
etc. Property (TL0) follows from ($\mu$1), and   (TL1--3) are immediate from the construction.
Also, in view of (MSG2) and  \equ{si ti},
 the $k$th
stage diameter $d_k$ is not bigger than 
$C e^{-\sigma 
(t_1 + k(\A + \B) + \A)}$,
hence (STL). Since Alice makes choices using her winning strategy, the limit
set $\ay$ of the collection must lie in $S$. Assumption ($\mu$2)
implies that $\Delta_k(\ca)$ is bounded below by a positive constant
$c$ independent of $k$ and $\B$.
\ignore{ Let $D \in \ca_k \cap
\cd_{s_k}$, let $D'_1, \ldots, 
D'_N \in \cd_{t_k}$ be 
Bob's essentially disjoint possible choices contained in $D$, as in (MSG1). Let
$A_1, \ldots, A_N$ 
be Alice's corresponding choices, so that $A_i \subset
D'_i,\, A_i \in \cd_{s_{k+1}}$ for $i=1, 
\ldots, N$. Let $p$ and $D_1, \ldots, D_N \in
\cd_{t_{k+1}-p}$ be as in assumption (MSG1), so that 
$D \subset \bigcup D_i.$
By the strong Federer assumption, there are positive $c_1, c_2,$
depending only on $\A$ and $p$ respectively, such that 
$$\mu(A_i) \geq c_1\mu(D'_i) \ \ \mathrm{and} \ \mu(D'_i) \geq c_2
\mu(D_i).$$
Therefore 
\[
\begin{split}
\mu \left(\bigcup_i A_i\right ) & \geq c_1 
\mu \left(\bigcup_i D'_i\right ) 
 = c_1 \sum_i
\mu \left(D'_i \right ) \\ & \geq c_1c_2 \sum_i
\mu \left(D_i \right ) \geq c_1c_2 \mu(D).
\end{split}
\]
Since $\ca_{k+1}(B) = \left\{A_1, \ldots, A_N \right\}$ we obtain
$\Delta_k(\ca) \geq c_1c_2.$

Finally,
 every parent contains at least $N(\B)$ children, and, by assumption (MSG1), each 
has measure $e^{-\delta (\A+\B)}$ times the measure of the parent. Hence
$\Delta_i(\ca) \ge N(\B)e^{-\delta (\A+\B)}$ for each $i$, and, by
}
Applying  Lemma \ref{lem: Urbanski} 
we find
\begin{equation*}
\begin{aligned}
\dim (\ay \cap U)& \geq  \underline{d}_\mu(U) -\limsup_{k\to\infty}\frac{ (k+1)\big(\log c\big)}
{\log\, C - \sigma (t_1 + k(\A + \B) + \A)
}\\
& 
=   \underline{d}_\mu(U)
 +
\frac1\sigma\left({ \frac{\log c}{\A+\B}  }\right) \to_{\B \to \infty}
\underline{d}_\mu(U)\,.
\end{aligned}
\end{equation*}
\end{proof}

\section{Games induced by contracting automorphisms}\name{sec: contr}
\subsection{Definitions}\name{def contr}
In this section we take $E = H$ to be a connected 
Lie group with a right-invariant Riemannian metric $d$, and assume that it admits 
a one-parameter group of 
automorphisms $
\{{\Phi}_t: t \in\R\}$ such that $\Phi_t$ is contracting for $t > 0$ 
(recall that ${\Phi}:H\to H$ is {\sl contracting\/} if 
for every $g\in H$, ${\Phi}^k(g)\to e$ as $k \to \infty$). It is not hard to see that 
$H$
must be simply connected and nilpotent, 
and the differential of each ${\Phi}_t$, $t > 0$, 
must be a linear isomorphism of
the Lie algebra $\goth h$ of $H$ with the modulus of all  eigenvalues strictly less than $1$. 
%
In other words, ${\Phi}_t = \exp(tX)$ where $X \in \End ( \goth h)$ and the real parts of all eigenvalues of $X$ are negative. Note that
$X$ is not assumed to be diagonalizable,  although this will be the case in 
our main example.
\commargin{removed `all the 
applications of the present paper', kind of a strong prediction}  

Say that a subset $D_0$ of $H$  is {\sl admissible\/} if it
is  compact and has non-empty interior.
For such $D_0$ and any $t \in\R$ and $x\in H$, 
define 
\eq{def psi}{\psi(x,t) =  {\Phi}_{t}(D_0)x\,,
}
and then introduce a partial ordering on $\Omega \df H \times \R$ by
\eq{def order}{
(x',t')\le (x,t)\quad\iff\quad\psi(x',t')\subset\psi(x,t)\,.
}
Monotonicity of $\psi$ is immediate from the definition, and we claim that, 
with $t_* = -\infty$ and some $a_*$, it satisfies conditions (MSG0--2).
Indeed, let $\sigma>0$ be any number such that  the real
parts of all the eigenvalues of $X$ are  smaller than $-\sigma$. 
Then, since $D_0$ is bounded,  it follows that for some $c_0 > 0$ 
one has 
\eq{dist squeeze}{d\big({\Phi}_t(g),{\Phi}_t(h)\big) \le c_0 e^{-\sigma t}
}
for all $g,h\in D_0$, thus (MSG2) is satisfied
(recall that the metric is chosen to be right-invariant, so all the elements of $\psi(\Omega_t)$ are isometric to
${\Phi}_{t}(D_0)$).
For the same reasons, for any open  $U\subset H$ there exists $s = s(U) > 0$ such that $U$ contains
a translate of  ${\Phi}_{t}(D_0)$ for any $t \ge s$, which implies (MSG1). Since $D_0$ is assumed to have
nonempty interior, (MSG0) follows as well, with $a_* = s(\Int D_0)$. 

 We denote $\mathcal{F} \df \{{\Phi}_t : t > 0\}$ and refer to the game determined by 
\equ{def psi} and \equ{def order} 
as 
 the  modified Schmidt game {\sl induced\/} by $\mathcal{F}$. Note that in this situation
the map $\psi$ is injective, i.e.\ the pair $(x,t)$ is uniquely determined by $D_0$ and the translate ${\Phi}_t(D_0)x$.
Consequently, without loss of generality we can describe the game in the language of choosing translates of
${\Phi}_\A(D)$ or ${\Phi}_\B(D)$ inside $D$, where $D$ is a domain chosen at some stage of the game. Clearly
when $H = \R^n$, $D_0$ is a closed unit ball and ${\Phi}_t = e^{-t}\text{Id}$, we recover Schmidt's original game.


  Note also that we have
suppressed $D_0$ from the notation. This is justified in
light of the following proposition: 

\begin{prop}
\name{prop: initial domain} Let $D_0, D_0'$ be admissible, and  define 
 $\psi$ and $\psi'$ as in \equ{def psi} using $D_0$ and $D_0'$
respectively. Let
$s>0$ be such that 
for 
some $x,x'\in H$,
\eq{squeeze}{
 {\Phi}_{s}(D_0)x\subset D_0' 
\text{ and }  {\Phi}_{s}(D_0')x'\subset D_0
}
(such an $s$ exists in light of \equ{dist squeeze} and the admissibility of $D_0, D_0'$). 
Suppose that  $\B > 2s$ and $S\subset H$ is $(\A, \B)$-winning for the $\psi$-MSG;
then it is $(\A + 2s,\B - 2s)$-winning for the $\psi'$-MSG. In particular, if 
 $S$ is $a$-winning for the $\psi$-MSG, then it is $(\A+
2s)$-winning for the $\psi'$-MSG. 
\end{prop}

\begin{proof} We will show how, using an existing $(\A, \B)$-strategy 
of the $\psi$-MSG, one can define an  $(\A + 2s,\B - 2s)$-strategy for the $\psi'$-MSG. 
Given a translate $B'_{k}$ of ${\Phi}_{t}(D_0')$ chosen by Bob at the $k$th step, 
pick a translate $B_{k}$  of ${\Phi}_{t+s}(D_0)$ contained in
it, and then, 
according to the given $\psi$-winning strategy, a translate $A_{k}$ of  
${\Phi}_{t+s + \A}(D_0)$. In the latter one can
find a translate of ${\Phi}_{t+2s + \A}(D_0')$; 
this will be the next choice $A'_{k}$ of Alice. Indeed, any move that could be made by 
Bob in response, that is, a translate $B'_{k+1}$ of 
${\Phi}_{t+2s + a+ b - 2s}(D_0') = {\Phi}_{t+ \A+ \B }(D_0')$ 
inside $A'_{k}$, 
will contain a translate of  ${\Phi}_{t+ \A+ \B + s}(D_0)$, and the latter can be viewed as a move 
responding to $A_{k}$ according to the
$\psi$-strategy. Thus the process can be continued, 
eventually yielding a point from $S$ in the intersection of all the chosen sets. 
\end{proof}

\subsection{A dimension estimate}\name{dimest contr}
Choose  a 
Haar measure $\mu$ on $H$ (note that $\mu$ is  both left- and right-invariant since $H$ is unimodular). Our
next claim is that conditions 
($\mu$1--2) are also satisfied. Indeed, since $D_0$ is admissible, $\mu(D_0)$ is positive, and one has
 \eq{measures}{\mu\big({\Phi}_t(D_0)x\big) = e^{-\delta t}\mu(D_0)\quad\text{for any }x\in H,\,t \in\R\,,} hence ($\mu$1),
where 
$\delta = - \Tr(X)$.
\ignore{Using
Jordan decomposition of $X$ one easily finds that there is a norm $\|
\cdot \|$ on $\R^n$ such that 
\eq{eq: norm squeeze}{\forall x \in \R^n, \ \ \|{\Phi}_{-t}(x)\| \leq
e^{-\lambda t} \|x\|,
}
so (ii) is clear. By admissibility of $D_0$, there is $a_*$ such that
for all $t \geq a_*$, $D_0$ contains 
a translate of ${\Phi}_{-t}(D_0)$. Then 
(i) is satisfied with this value of $a_*$. }
Also, in view  of  \equ{measures} and the definition of $\psi$, to verify ($\mu$2) it suffices to show

\begin{lem}\name{prop: msg3} Let $D_0$ be admissible and let $a_*$ be as in {\rm (MSG0)}. 
Then there exists
$\bar c > 0$ such that for any $b > a_*$, $D_0$ 
 contains essentially disjoint right translates
$D_1, \ldots, D_N$ of ${\Phi}_b(D_0)$  such
that 
\eq{density}{\mu\big(\bigcup_i D_i\big) \geq \bar c \mu(D_0)\,.}
\end{lem} 

Indeed, if this holds, then, in view of  \equ{measures} the conclusion of the lemma
holds with $D_0$ replaced by ${\Phi}_t(D_0)x$ for every $x$ and $t$, and then, by 
(MSG0) and  \equ{measures}, ($\mu$2) holds with $c = \bar c e^{-\delta a}$.
\smallskip

For the proof  of Lemma \ref{prop: msg3}  we use the following result from \cite{KM}:

\begin{prop}\name{lem: tess} Let $H$ be a connected simply connected
nilpotent Lie group. Then for any $r>0$ 
there exists a neighborhood $V$ of
identity in $H$ 
with piecewise-smooth boundary and 
with $\diam(V)<r$,
and a countable subset $\Delta\subset H$ such that $H = \bigcup_{\gamma\in\Delta}\overline{V}\gamma$ and \eq{disj}{V\gamma_1  \cap V\gamma_2 = \varnothing\text{ for different }\gamma_1,
\gamma_2 \in \Delta\,.}
\end{prop} 

For example, if $H = \R^n$ one can take $V$ to be the unit cube,
$$V = \big\{(x_1,\dots,x_n): |x_j|<1/2\big\}\,,$$ and $\Delta = \Z^n$, or rescale both $V$ and $\Delta$
to obtain domains of arbitrary small diameter.
See \cite[Proposition 3.3]{KM} for a proof of the above proposition.

\begin{proof} [Proof of Lemma \ref{prop: msg3}] First note that, since $b$ is assumed to be
greater than $a_*$, $D_0$ 
 contains  at least one translate  of ${\Phi}_b(D_0)$, thus the left hand side of \equ{density}
 is not less than  $e^{-\delta b}\mu(D_0)$. Now, in view of  \equ{measures}, while proving the
lemma one can replace $D_0$ by ${\Phi}_t(D_0)$ for any $t\ge 0$. Thus without loss of generality one
can assume that $D_0$ is contained in $V$ as in Proposition \ref{lem: tess} with $r \le 1$, and that
\equ{dist squeeze} holds $\forall\,g,h\in V$. Now choose a nonempty open ball $B\subset D_0$.
We are going to estimate from below the measure of the union of sets of the form ${\Phi}_b(D_0\gamma)$, where
$\gamma \in \Delta$, 
contained in $B$; they are disjoint in view of 
\equ{disj}. 

Since
\begin{equation*}
\bigcup_{\gamma \in \Delta,\,{\Phi}_b(\overline V\gamma)\subset B}{\Phi}_b(\overline V\gamma)
=
\bigcup_{\gamma \in \Delta,\,{\Phi}_b(\overline V\gamma)\cap B\ne\varnothing}{\Phi}_b(\overline V\gamma)
\ \ \ssm \bigcup_{\gamma \in \Delta,\,{\Phi}_b(\overline V\gamma)\cap \partial B\ne\varnothing}{\Phi}_b(\overline V\gamma)
\,,\end{equation*}
we can conclude that the measure of the set in the left hand side is not less than
$$
\mu(B) - 
\mu\left(\big\{ \diam\big({\Phi}_b( V)\big)\text{-neighborhood of }\partial B\big\}\right).$$
Clearly for any $0 < \vre < 1$  the measure of the  $\vre$-neighborhood of $\partial B$ is bounded from above by $c'\vre$
where $c'$ depends only on $B$. In view of \equ{dist squeeze} and \equ{measures}, it follows that
\begin{equation*}
\begin{split}
\mu\left(
\bigcup_{\gamma \in \Delta,\,{\Phi}_b(D_0\gamma)\subset B}{\Phi}_b(D_0\gamma)\right) 
&\ge \frac{ \mu(D_0)}{ \mu(V)} \left(\mu(B) - c_0 c' e^{-\sigma b} \diam(V)\right)\\ & = \mu(D_0)\left(\frac{ \mu(B)}{ \mu(V)} - \frac{ c_0 c'  \diam(V)}{ \mu(V)} e^{-\sigma b}\right),
\end{split}
\end{equation*}
which is not less than $\frac{ \mu(B)}{ 2\mu(V)} \mu(D_0)$ if $e^{-\sigma b} \le{ \mu(B)}/{ 2c_0c' \diam(V)}$.
Combining this with the remark made in the beginning of the proof, we conclude
that \equ{density} holds with $$\bar c = \min\left( \frac{ \mu(B)}{ 2\mu(V)}, \left(\frac{ \mu(B)}{ 2c_0c' \diam(V)}
\right)^{\delta/\sigma}\right).
$$
\end{proof}

In view of the discussion preceding Proposition \ref{prop: msg3},
an application of Theorem \ref{thm: WD} yields
\begin{cor}\name{cor: msg3}
Any  winning set for the MSG induced by $\mathcal{F}$ as above is thick. 
\end{cor} 

\subsection{Images of winning sets}\name{images contr} 
One of the nice features of the original Schmidt's game is the stability of the
class of its winning sets under certain maps, see e.g.\ \cite[Theorem 1]{Schmidt games}
or \cite[Proposition 5.3]{Dani survey}. We close this section 
by describing some self-maps of $H$ which send winning sets of the
game induced by $ \FF$ to winning sets:

\commargin{I left just one combined statement, which is likely to be used in the future
for the bounded orbits business; please check the proof!}
\begin{prop}\name{prop: affine commuting} Let $\varphi$ be an  automorphism of $H$
commuting with $\Phi_t$ for all $t$. Then  there exists $s > 0$ (depending on $\varphi$ and the choice
of an admissible $D_0$) such that the following holds. Take $t_0\in\R$ and $x_0\in H$, 
and consider 
\eq{def f}{
f:H\to H, \ x\mapsto \Phi_{t_0}\big( \varphi (x)\big)x_0\,.
}
\commargin{Maybe it is worth mentioning that
some mysterious more general classes of maps can be considered,
or maybe it is clear from the proof anyway}

Then for any $a > a_*$, $b > a_* + 2s$ and 
any $S\subset H$ which is $(a,b)$-winning for the MSG induced by $\FF$, the set
$f(S)$ is $(a+2s, b-2s)$-winning for the same game.
\end{prop}

\begin{proof} Since $D_0$ is admissible and $\varphi$ is a homeomorphism, 
there exists $s > 0$ such that some translates of both $\varphi(D_0)$ and $\varphi^{-1}(D_0)$
contain $\Phi_s(D_0)$.   Then, for  $f$ as in \equ{def f}, since $\varphi$ is a group homomorphism
and $\Phi_t\circ \varphi = \varphi \circ \Phi_t$ for all $t$,  it follows that
\eq{bilipschitz}{ 
\begin{aligned}\text{for any }t \in \R\text{ and }x \in H \quad\exists\, x', x'' \in 
H\text{ such that}\quad\qquad\\
f\big(\Phi_{t}(D_0)x \big)\supset \Phi_{t+t_0+s}(D_0)x' ,\ 
f^{-1}\big(\Phi_{t}(D_0)x \big)\supset
\Phi_{t-t_0+s}(D_0)x''\,. 
\end{aligned}
}
Suppose that Alice and Bob are playing the game with  parameters $(\A+2s,
\B-2s)$ and target set  $f(S)$. Meanwhile their clones  $\til{\text{Alice}}$ 
and  $\til{\text{Bob}}$ are playing  with parameters
$(\A,\B)$, and 
we are given a strategy for $\til{\text{Alice}}$ to win on $S$.
Let $B_k = \Phi_t(D_0)x$ be a move made by Bob
at the $k$th stage of the game. Thus by
\equ{bilipschitz}, 
$f^{-1}(B_k)$ contains a set $\til B_k = \Phi_{t-t_0+s}(D_0)y$ for some $y\in H$. 
Then, in response to $\til 
B_k$ as if it were  $\til{\text{Bob}}$'s choice, $\til{\text{Alice}}$'s strategy 
specifies $\til A_k=\Phi_{t-t_0+s+a}(D_0)y' \subset \til B_k$, a move which ensures convergence to a point of $S$. 
Again by \equ{bilipschitz}, the set $f(\til A_k)$
contains $A_k = \Phi_{t+a+2s}(D_0)x'$ for some $x'\in H$, which can be chosen by Alice as her next move. Now for any choice made by Bob of $$B_{k+1} = \Phi_{t+a+2s+ b-2s}(D_0)z= \Phi_{t+a+ b}(D_0)z\subset A_k$$
Alice can proceed as 
above, since $f^{-1}(B_{k+1})$  will contain a valid  
 move for $\til{\text{Bob}}$ in response to   $\til A_k$. Continuing this way, Alice can enforce
$$\bigcap_{k = 1}^\infty  A_k = \bigcap_{k = 1}^\infty   f(\til A_k) = f\left(\bigcap_{k = 1}^\infty   \til A_k\right)\in f(S)\,,$$
winning the game.
\end{proof}

\ignore{
Putting these together we obtain:

\begin{cor}\name{cor: putting together}
If $f_1, f_2, \ldots$ are $(\FF, D_0, s)$-bilipschitz (for some fixed
$s$), and $g_1, g_2, 
\ldots$ are $\psi$-dilations for $\psi$ as in \equ{def psi}, then for any
winning set $S$, $\bigcap_i f_i \circ g_i (S)$ is also winning.  

\end{cor}
}

\section{$\BA(\vr)$ is winning}\name{sec: win}
In this section we take $H = \R^n$ and prove Theorem \ref{thm: main},
 that is, exhibit a strategy
for the MSG induced by $\mathcal{F}^{(\vr)}$ as in \equ{defn ar} which ensures that
Alice can always zoom to a point in $\BA(\vr)$. 
Our argument is similar
to that from \cite{PV-bad}, which in turn is based on 
 ideas of Davenport.
In view of remarks made at the end of the previous section, we can make an arbitrary choice 
for the initial admissible domain $D_0$, and will choose it to be the unit cube in $\R^n$, so that
translates of ${\Phi}^{(\vr)}_t(D_t)$ are  boxes with sidelengths $e^{-(1+r_1)t},\dots,e^{-(1+r_n)t}$.
The main tool will be the so-called `simplex lemma', the idea of which is attributed
to Davenport in  \cite{PV-bad}. Here is a version suitable for our purposes. 

\begin{lem}
\name{lem: simplex} Let $D\subset \R^n$ be a  box  with sidelengths
$\rho_1,\dots,\rho_n$, and for $N > 0$ denote by $\,\mathcal{Q}(N)$
the set of rational vectors $\p/q$  
written in lowest terms with $0 < q  < N$. Also let  ${f}$ be a
nonsingular affine transformation  of $\R^n$ and 
$J$ the Jacobian of ${f}$ (that is, the absolute value of the determinant of
its linear part). Suppose that \eq{smallvolume}{\rho_1\cdots\rho_n <
\frac{J}{n!N^{n+1}}\,.} 
Then there exists an affine hyperplane $\mathcal{L}$ such that
${f}\big(\mathcal{Q}(N)\big)\cap D \subset\mathcal{L} $. 
\end{lem}

\begin{proof} Apply \cite[Lemma 4]{KTV} to the set ${f}^{-1}(D)$.
\ignore{The claim holds trivially if $\#(\mathcal{Q}(N)\cap D) \le n$. Otherwise, denote $n+1$
different elements of $\mathcal{Q}(N)\cap D$ by $\p^{(0)}/q^{(0)},\dots,\p^{(n)}/q^{(n)}$,
assume they do not lie on a single affine hyperplane, 
and let $\Delta$ be the $n$-dimensional simplex subtended by them. Then  $n!m(\Delta)$  is equal to 
the absolute value of the determinant
\eq{det}{
\left|\begin{matrix}1 & p^{(0)}_1/q^{(0)} & \dots & p^{(0)}_n/q^{(0)}
\\1 & p^{(1)}_1/q^{(1)} & \dots&  p^{(1)}_n/q^{(1)} \\   \vdots&
\vdots& \vdots& \vdots\\1 & p^{(n)}_1/q^{(n)} & \dots &
p^{(n)}_n/q^{(n)} \end{matrix}\right|  
 \,.
}
By assumption, the above determinant is nonzero; hence its absolute value is at least
$1/q_0\cdots q_n > N^{-(n+1)}$. Since $\Delta \subset D$, we get
$
N^{-(n+1)} < n!m(D) = n!\rho_1\cdots\rho_n$, contradicting \equ{smallvolume}.}
\end{proof}

Now let us state a strengthening of Theorem \ref{thm: main}:

\begin{thm}\name{thm: precise}
Let  $\vr$ be as in \equ{defn r},  ${\Phi}_t = {\Phi}^{(\vr)}_t$ as in
\equ{defn ar}, and $\psi$ as in \equ{def psi} where $D_0 = [-1/2,1/2]^n$.
 Then for any nonsingular affine transformation ${f}$ of $\R^n$
whose linear part commutes with ${\Phi}_1$ and any $a > \max_i\frac{\log 2}{1 + r_i}$,
${f}\big(\BA(\vr)\big)$ 
is an $a$-winning set for the $\psi$-MSG.
\end{thm}

\commargin{I decided against a special corollary, but can put it back if you want}
We remark that in this case $\A$ can be chosen 
independently of the linear part of $f$; note that this is not guaranteed by 
a general result as in Proposition  \ref{prop: affine commuting}, but relies on special properties 
of the set $\BA(\vr)$. Consequently, in view of 
Theorem \ref{thm: countable general}, for any sequence $\{L_i\}$ of nonsingular diagonal 
matrices and a sequence $\{\vy_i\}$ of vectors in $\R^n$, the intersection
$\bigcap_{i=1}^\infty \left(L_i\big(\BA(\vr)\big) + \vy_i\right)$
is also $\A$-winning, hence thick.

\begin{proof}%
[Proof of Theorem \ref{thm: precise}]
 We first claim that
$\vy \in {f}\big(\BA( \vr)\big) $ if and 
only if there is $c'>0$ such that 
\eq{shifted}{\max_{1 \leq i \leq n} \left|y_i-{f}\left(\frac{\p}{q}
\right)_i
\right| \geq \frac{c'}{q^{1+r_i}} 
}
for all $\p \in \Z^n$ and $q \in \N$ (here ${f}(\mathbf{x})_i$
denotes the $i$th coordinate of ${f}(\mathbf{x})$). 

To see this,
let $\vr = 
(r_i)$ be as in \equ{defn r}, and let 
\eq{decomposition}{\R^n = \bigoplus V_j}
be the eigenspace decomposition for ${\Phi}_1$. Letting $\{\E_i\}$ denote the
standard basis of $\R^n$, we have that $V_j = \spa \left( \E_i
:i \in I_j\right)$ where $I_j$ is a maximal subset of $\{1, \ldots,
n\}$ with $r_i$ 
the same for all $i \in I_j$. 
Since the
linear part of ${f}$ commutes with ${\Phi}_1$, it preserves each
$V_j$, so that we may write 
$${f}(\x) = \x_0 + \sum_j A_jP_j(\x),$$
where $\x_0 \in \R^n$, $P_j$ is the projection onto $V_j$ determined
by the direct sum decomposition \equ{decomposition}, and $A_j : V_j
\to V_j$ is an invertible linear map. Let $K$ be a positive constant
such  that for each $\x \in V_j$,
$$\frac{1}{K} \|\x \| \leq \|A_j \x\| \leq K\|\x\|.$$
With these choices it is easy to show that 
\equ{defn rba} 
implies \equ{shifted} for $\vy = {f}(\x)$ with $c' = c^{\max r_i}/K$, and
similarly that \equ{shifted} for $\vy = {f}(\x)$ implies
\equ{defn rba} 
with $c= \left({c'}/{K} \right)^{1+ \max r_i^{-1}}.$

Let us fix $a > \max_i\frac{\log 2}{1 + r_i}$ and ${t_0} > 0$ such that
\eq{t1}
{e^{-{t_0}(n+1)} <  \frac J {2^n n!} \,,}
where 
$J$ is the Jacobian of ${f}$. We will specify a strategy for 
Alice. 
Bob
makes a choice of (arbitrarily large) $b$ and an initial rectangle, that
is, a translate $B_1$  of ${\Phi}_{t_1}(D_{0})$ 
for some $t_1$ which we demand to be at least ${t_0}$ (the latter is justified by Lemma \ref{lem: dummy}).
We then choose a positive constant $c'$ such that
\eq{c1}{e^{(a+b)(1+r_i)}c' < 
\left
(\tfrac12 -
e^{-a(1+r_i)}
\right)e^{-t_0(1+r_i)}
} 
for each $i$ (we remark that $\tfrac12 - e^{-a(1+r_i)}$ is positive because of the choice of $a$).
Our goal will be to prove the following

\begin{prop}\name{prop: induction} For any choices of $B_1,\dots,B_k$ made by Bob
it is possible for Alice to choose $A_{k}\subset B_{k}$ such that whenever
$\vy\in A_{k}$, inequality \equ{shifted} 
holds 
 for all $\p,q$ with $0 < q < e^{(k-1)(a+b)}$.
\end{prop}

If the above claim is true, then the intersection point $\vy$ of all balls will satisfy \equ{shifted}
for all $\p\in\Z^n$ and $q\in\N$, that is, will belong to ${f}\big(\BA(\vr)\big)$.\end{proof}

\begin{proof}[Proof of Proposition \ref{prop: induction}] We proceed
by induction on $k$. In case $k = 1$, the statement is trivially true since 
the set of $q\in\N$ with $0 < q < 1$ is empty. Now suppose that $A_1,\dots,A_{k-1}$ 
are chosen according to the claim, and Bob picks $B_k\subset A_{k-1}$. Note that
$B_k$ is a box with sidelengths $\rho_i \df e^{-(t_1 + (k-1) (a + b))(1+r_i)}$, $i = 1,\dots,n$. 


By induction, for each $\vy\in B_k\subset A_{k-1}$ \equ{shifted} holds
for all $\p,q$ with $0 < q < e^{(k-2)(a+b)}$.
Thus we need to choose $A_{k}\subset B_{k}$ such that the same is
true for \eq{new}{ e^{(k-2)(a+b)}\le q < e^{(k-1)(a+b)}\,.} 
Let  $\p/q$, 
written in lowest terms, be such that 
\equ{new} holds, and  that \equ{shifted} does not hold
for some $\vy\in B_k$; in other words, for each $i$ one has
$$|y_i- {f}(\p/q)_i|<
\frac{c'}{q^{1+r_i}} \le
\frac{e^{(a+b)(1+r_i)}
c'}{e^{(k-1)(a+b)(1+r_i)}}$$ 
for some $\vy\in B_k$. Denote by $\tilde \vy$ the center of $B_k$, so that 
$$|y_i - \tilde y_i| \le \frac{\rho_i}2 
\,;$$ then for each $i$,
\begin{equation*}
|\tilde y_i- {f}(\p/q)_i| <   \left(e^{(a+b)(1+r_i)}c' +
e^{-t_0(1+r_i)}/2\right)e^{-(k-1)(a+b)(1+r_i)}
\under{
\equ{c1}}< \rho_i
\,. 
 \end{equation*}
 Thus, if we denote by $D$ the box centered at $\tilde \vy$ with
sidelengths 
$2\rho_i$, 
 we can conclude that ${f}(\p/q)\in D$; but also $\p/q\in \mathcal{Q}(e^{(k-1)(a+b)})$, and 
 $$2^n\rho_1\cdots\rho_n  = 
2^n  e^{-\left(t_0 + (k-1)(a+b)\right)(n+1)} \under{
\equ{t1}}< \frac J {{ n! }(e^{(k-1)(a+b)})^{n+1}} \,.$$ 
 Therefore, by Lemma \ref{lem: simplex}, there exists an affine hyperplane $\mathcal{L}$ containing all 
${f}(\p/q)$ as above.

\ignore{

By induction, for each $\x\in B_k$ \equ{shifted} holds
for all $\p,q$ with $0 < q < e^{(k-1)(a+b)}$.
Thus we need to choose $A_{k+1}\subset B_{k}$ such that the same is
true for $ e^{(k-1)(a+b)}\le q < e^{k(a+b)}$. 
Let  $\p/q$, 
written in lowest terms, be such that $ e^{(k-1)(a+b)}\le q < e^{k(a+b)}$ and  that \equ{shifted} does not hold
for some $\x\in B_k$, in other words, for each $i$ one has
$$|x_i- {f}(\p/q)_i|<
\frac{c^{r_i}}{q^{1+r_i}} \le
\frac{e^{(a+b)(1+r_i)}c^{r_i}}{e^{k(a+b)(1+r_i)}}$$ 
for some $\x\in B_k$. Denote by $\tilde \x$ the center of $B_k$, so that 
$$\max_i|x_i - \tilde x_i| \le \frac12 e^{-(t_0 + k (a + b))(1+r_i)}\,;$$ then for each $i$,
\begin{equation*}
\begin{split} |\tilde x_i- p_i/q| &<   \left(e^{(a+b)(1+r_i)}c^{r_i} +
e^{-t_0}/2\right)e^{-k(a+b)(1+r_i)}\\ &\under{\equ{t1}, \equ{c1}}<
\frac3{8(n!J)^{1/n}}{e^{-k(a+b)(1+r_i)}}\,. 
\end{split}
 \end{equation*}
 Thus, if we denote by $D$ the box centered at $\tilde \x$ with
sidelengths $$\rho_i =  \frac3{4(n!J)^{1/n}}e^{-k(a+b)(1+r_i)}\,,$$ 
 we can conclude that $\p/q\in D$; but also $\p/q\in \mathcal{Q}(e^{k(a+b)})$, and 
 $$\rho_1\cdots\rho_n  = \big(\tfrac34\big)^n \frac1{ n!J }
e^{-k(a+b)(n+1)} < 1/{ n!J } \big(e^{k(a+b)}\big)^{n+1} \,.$$ 
 Therefore, by Lemma \ref{lem: simplex}, there exists an affine hyperplane $\mathcal{L}$ containing all 
$\p/q$ as above. 
}

Clearly it will be advantageous for Alice is to stay as far from all those vectors
as possible, i.e., choose $A_{k}\subset B_{k}$ to be a translate of ${\Phi}_{t_1 + (k-1) (a + b) + a}(D_0)$
which maximizes the distance from  $\mathcal{L}$.  A success is guaranteed by the assumption
$a > \max_i\frac{\log 2}{1 + r_i}$, which amounts to saying that for each $i$, the ratio of the 
length of the $i$th
side of the new box to the length of the $i$th
side of $B_k$ is  $e^{-a(1+r_i)} < 1/2$. This implies that for each $\x\in A_{k}$ chosen this way
and any $\x'\in \mathcal{L}$, there exists $i$ such that
$
|x_i - x'_i|$ is not less than the length of the $i$th
side of $B_k$ times $(\frac12  - e^{-a(1+r_i)})$. Therefore, whenever   
\equ{new} holds and $\x\in A_{k}$, for some $i$ one has
\begin{equation*}
\begin{split} 
|x_i - {f}(\p/q)_i| &\ge e^{-(t_0 + (k-1) (a + b))(1+r_i)}\big(\tfrac12  -
e^{-a(1+r_i)}\big)\\ & = e^{-t_0(1+r_i)}\big(\tfrac12  -
e^{-a(1+r_i)}\big)e^{-(a + b)(1+r_i)}e^{-( (k-2) (a + b))(1+r_i)}\\
&\ge  e^{-t_0(1+r_i)}\big(\tfrac12  - e^{-a(1+r_i)}\big)e^{-(a +
b)(1+r_i)} q^{-(1 + r_i)} \under{\equ{c1}}\ge  
c' q^{-(1 + r_i)}\,,
\end{split}
 \end{equation*}
 establishing \equ{shifted}.
\end{proof}

\ignore{Combining this with Propositions \ref{prop: general dilations} and
\ref{prop: general bilipschits} we obtain:

\begin{cor}\name{cor: for intro}
Let $f_1, f_2, \ldots$ be a sequence of invertible maps $\R^n \to
\R^n$ of the form $f_i = A_i \circ B_i \circ C_i$, where:
\begin{itemize}
\item
Each $A_i$ is an $(\FF^{(\vr)})$-dilation. 
\item
Each $B_i$ is $(\FF^{(\vr)}, D_0, s)$-bilipschitz for some fixed $s$.  
\item
Each $C_i$ is an affine map whose linear part commutes with $\Phi_t$
for all $t$;
\end{itemize}
Then $\dim \, \bigcap_i {f}_i\big(\BA(\vr)\big) = n.$ 

\end{cor}
}


\section{Concluding remarks}\name{sec: next}

\subsection{Dimension of winning sets for Schmidt's  game}\name{sec: fulldim}
The formalism developed in \S\S\ref{sec: games}--\ref{sec: contr}
appears to be quite general, and we expect it to be useful in a
wide variety of situations. In particular, new information can be
extracted even for the original Schmidt's game.
Namely, here we state a condition on a metric space 
sufficient to conclude that any winning set of  Schmidt's game \equ{rn}
has big enough dimension.
This will be another application of Theorem \ref{thm: WD}.
Recall that a locally finite  Borel measure $\mu$ on a
metric space $X$ is called {\sl Federer\/}, or doubling, if there is $K>0$ and $\rho > 0$ such that for all
$x \in \supp\,\mu$ and $0 < r < \rho$, 
\eq{defn Federer}{
\mu\big({B}(x,3r) \big) \leq K \mu\big({B}(x,r)
\big)\,.
}
\begin{prop}
\name{cor: Federer}
Let $E$ be a complete metric space which is the support of a Federer
\commargin{added $\alpha,\beta$}
measure  $\mu$. Then there exist  $c_1,c_2 > 0$, depending only on
$K$ as in \equ{defn Federer},
such that whenever  $0 < \alpha  < 1$,  $0 < \beta < 1/2$, and $S$ is an $(\alpha,
\beta)$-winning set for  Schmidt's game 
as in
\equ{rn}  played on $E$ and  $ \varnothing \ne U \subset E$ is open,
one has
\eq{federer estimate}{\dim (S \cap
U) \geq \underline{d}_{\mu}(U)  -
 \frac{c_1|\log \alpha|  + c_2}{|\log \alpha|+ |\log \beta|}  \,.} 
In particular,   $\dim (S\cap U) \geq
\underline{d}_{\mu}(U)$ if $S$ is winning. 
\end{prop}

Clearly Theorem \ref{thm: full dim} is a special case of
the above result. In addition, Proposition \ref{cor: Federer} 
generalizes a recent result of L.\ Fishman \cite[Thm.\ 3.1 and Cor.\ 5.3]{Fishman} that for a
measure $\mu$ satisfying a power law (see \equ{pl}; this condition obviously implies Federer)
a winning set for Schmidt's
original game \equ{rn} played on $E = \supp \, \mu$ has full
\hd. 
See \cite[Example 7.5]{bad} for an example of a
\commargin{closed is automatic} 
subset of $\R$ (a similar construction is possible in $\R^n$ for any $n$) \commarginnew{changed}
supporting a measure of full \hd\ which is Federer but does not satisfy a power
law. 

\ignore{It follows that the metric space constructed  in Proposition \ref{prop: lower dim example}
cannot support a Federer measure of positive dimension.}

\begin{proof}[Proof of Proposition  \ref{cor: Federer}] 
We need to check the assumptions of Theorem  \ref{thm: WD}.  Conditions (MSG0--2) are immediate, and  ($\mu$1) 
holds since $\supp\,\mu = E$.  Thus it only suffices
to verify ($\mu$2). 
It will be convenient to switch back to Schmidt's multiplicative notation of \S\ref{special}.
Fix $0 < \alpha  
< 1$; we claim that there exists $c' > 0$
such that for any $x,x'\in E$ and $0 < r < \rho$ one has
\eq{strong federer}{{B}(x', \alpha r) \subset {B}(x, r) \implies \mu\big({B}(x' ,\alpha r)  \big) \geq c'
\mu\big({B}(x ,r)  \big)\,.}
Indeed,  choose $m\in\N$ and $c' > 0$ such
that \eq{def const}{\alpha/6 < 3^{-m} \leq \alpha/2\text{ and }c' = K^{-m}\,.}
Iterating 
\equ{defn Federer} $m$ times, we find  $\mu\big({B}(x' ,\alpha
r)\big) \geq c'\mu \big({B}(x',2r) \big)$ for any $x' \in E$ and
$r>0$. Since for any $x' \in
{B}(x, r)$ the latter ball 
 is contained in $ {B}(x', 2r)$, \equ{strong federer} follows. 

Now take $0 < \beta < 1/2$ and $\omega = (x,r)\in E\times \R_+$ with $r < \rho$, and 
let
$x_i,\  i=1, \ldots, N$, be a maximal collection of points such that $\theta_i \df (x_i,\beta r)\le_s \omega$ and
balls $ B(\theta_i)$ are pairwise disjoint.
By maximality, 
$$B\big(x,(1-\beta)r\big) \subset \bigcup_{i=1}^N {B}(x_i,
3\beta r)\,.$$
 (Indeed, otherwise there exists $y\in  B\big(x,(1-\beta)r\big)$ with $d(y,x_i) > 3\beta r$, 
which implies that $(y,\beta r)\le_s \omega$ and $ B(y,\beta r)$ is disjoint from $ B(\theta_i)$
for each $i$.)
In view of \equ{strong federer}, for any choices of 
$\theta_i' \df (x_i', \alpha \beta r) \le_s \theta_i $ one has
$\mu\big( B(\theta_i')\big) \geq 
c'\mu\big({B}(\theta_i)\big)$. This implies 
\begin{equation*}
\begin{aligned}
\mu\big (\bigcup  B(\theta_i') \big)  & = \sum \mu\big ( B(\theta_i) \big)  
 \geq c' \sum \mu
\big({B}(\theta_i) \big) \\
& \geq   \frac{c'}{K} \sum \mu
\big({B}(x_i, 3\beta r) \big) \geq  
\frac{c'}{K}\mu\big( B(x,(1-\beta)r\big)\\ &\geq  
\frac{c'}{K}\mu\big( B(x,r/2)  \big)\geq  
\frac{c'}{K^2}\mu\big( B(\omega) \big) \,.
\end{aligned}
\end{equation*}
Hence ($\mu$2) holds with $c = c'/K^2$,  and \equ{federer estimate}, with explicit
\commargin{We can compute $c_1 = \log K / \log 3$, $c_2 = \log K (3-\log 2/\log 3) $ but this is kind of boring..}
$c_1$ and $c_2$,  follows from \equ{wd}
and \equ{def const}.
\end{proof}

\ignore{We remark that, in view of \equ{wd} and an estimate
for $c$ obtained above,  whenever $S$ is  $(\alpha,\beta)$-winning  one has
$${\dim (S \cap
U) \geq \underline{d}_{\mu}(U)  -
 \frac{\log K \left(\frac{|\log \alpha/2|}{\log 3} + 3\right)}{|\log \alpha|+ |\log \beta|}  \,.} $$
\medskip}

\ignore{
Then one has  $\sigma = 1$, $\delta = n =
\underline{d}_\mu(x)$  
for all $x$, and $N(s) \ge c_n e^{ns}$ for some
constant $c_n$ dependent only on $n$. 
}

\ignore{We will say that $\mu$ is {\sl strong Federer with respect to $\cd$}
if for any $p>0$ there is $c=c(p)>0$ such that for all 
$t \geq t_*$, all $D
\in \cd_t$ and all $D' \in \cd_{t+p}$ with $D' \subset D$, one has 
$\mu(D') \geq c\mu(D).$ Note that this assumption implies that for any
$D \in \cd_t$ and any $p>0$, a maximal collection of disjoint elements
of $\cd_{t+p}$ contained in $D$ is finite. 

Specializing to the setup \equ{rn} of Schmidt's original game and
taking $p = \log 3$, this 
implies that there is $c>0$ such that if $B = {B}(x,3r)$ and
$B'={B}(z, r) 
\subset B$ then $\mu(B') \geq c\mu(B).$ This implies
the Federer condition and is
implied by the power law  
condition  (see \cite{bad}
for a discussion of conditions on measures). For
an example of a measure $\mu$ which is strongly Federer (with respect
to $\cd$ as in \equ{rn}) but does not satisfy a power law, see
\cite[Example 7.5]{bad}. }

The next proposition shows that, as was mentioned in \S\ref{dimest},
without additional assumptions on a metric space 
the conclusion
of Theorem \ref{thm: full dim} could fail:

\begin{prop}\name{prop: lower dim example}
There exists a complete metric space $E$ of positive \hd\ containing a countable
(hence zero-dimensional) winning set $S$ for
the game \equ{rn}.

\end{prop}

\begin{proof}
Let $X = \{0,1,2\}^{\N}$, equipped with the metric 
$$d\big((x_n), (y_n) \big) = 3^{-k}, \ \ \mathrm{where \ } k =
\min\{j: x_j \neq y_j\}.$$
Let $E \subset X$ be the subset of sequences in which the digit
0 can only be followed by 0; i.e. 
$$x_{\ell} =0, \ k \geq \ell \implies x_k =0. $$
Then $E$ is a closed subset of $X$ so is a complete metric space when
equipped with the restriction of $d$. 

Let $S$ be the set of sequences in $E$ for which
the digit 0 appears. Then $S$ is a countable dense subset in $E$ but no point
in $S$ is an accumulation point of $E \sm S$. In particular $\dim (S) =0$,
and it is easily checked that $\dim (E) = \log 2 / \log 3 >0.$

Let $\alpha = 1/27,$ and let $\beta$ be arbitrary. 
Suppose that  Bob chooses $\omega = (x, r)$,
where $x = (x_n)$.
Letting Alice
play arbitrarily we can assume that $r<1$. Let $\ell \in \N$ be
chosen so that $3^{-(\ell+1)} < r \leq 3^{-\ell}.$ 
Note that  $ B(\omega)$ contains all sequences $(y_n)$
with $y_i = x_i$ for all $i \leq \ell$, and in particular the sequence
$z = (x_1, \ldots, x_{\ell}, x_{\ell+1}, 0, 0, \ldots) \in S.$ Now Alice chooses 
$\omega' = (z, \alpha r)$; it is easy to see that $\omega'\le_s \omega$
and that $ B(\omega')=\{z\}$ (a singleton), since any
other sequence in this ball must begin with $(x_1, \ldots, x_{\ell},
x_{\ell+1}, 
0)$. Thus the outcome of the game is $z$ and Alice is the winner.
\end{proof}

It is not hard to see that such an example can be realized as a compact subset of 
$\R$ with the induced metric (e.g.\ by identifying sequences $(x_n)$ with real numbers
$0.x_1x_2\dots$ expanded in base $3$).


It is also worth remarking that another special case of our general framework is
an $(\alpha,\beta)$-game played on arbitrary metric space $E$ but with Schmidt's
containment relation  \equ{cont} replaced by 
  \eq{weakcont}									
{(x',r') \le (x,r)\quad \iff \quad  B(x',r') \subset   B(x,r)\,,	
}
similarly to the way it was done in \equ{def order}. The two conditions are equivalent when $E$ is a Euclidean space.
However in general, e.g.\ when $E$ is a proper closed  subset of $\R$ or $\R^n$
such as those considered in \cite{Fishman},  \equ{weakcont}	
is weaker, and the classes of winning sets for the two games could differ.
Still, by modifying the argument of this subsection one can show that
the conclusions of both propositions hold when the game is played according to 
the weaker containment relationship.

\subsection{Sets of the form \equ{tripleint} and their generalizations}\name{sec: int} 
Take $E = \R^n$ and let $\mathcal{F}$ be a one-parameter semigroup of its linear
contracting transformations. Suppose that $E = E_1\oplus E_2$ where both $E_1$
and $E_2$ are invariant under $\mathcal{F}$,  denote by $\mathcal{F}_1$
the restriction of $\mathcal{F}$ to $E_1$, and suppose that 
$S_1\subset E_1$ is a 
winning subset  of the MSG induced by $\mathcal{F}_1$. Then it immediately follows
from Lemma \ref{lem: product} that $S_1\times E_2$ is a 
winning subset  of the MSG induced by $\mathcal{F}$. Applying it to $\mathcal{F} = \mathcal{F}^{(\vr)}$
as in \equ{defn ar} we obtain

\begin{prop}\name{prop: intersections} For $\vr$ as in \equ{defn r} and $1\le k \le n$, define
$\vs\in\R^k$ 
by 
\eq{defn s}{
s_i = \frac{1 + (k+ 1) r_i - \sum_{l=1}^k r_l}{k +  \sum_{l=1}^k r_l}\,,\quad i = 1,\dots, k\,.
}
Then $
\BA(\vs) \times \R^{n-k}$
is a  winning set for the MSG induced by $\mathcal{F}^{(\vr)}$, and therefore so is its intersection with $\BA(\vr)$.
\end{prop}

\begin{proof}  Note that $\vs$ 
is defined
so that $\sum_i s_i$ 
is equal to $ 1$, and
the vector $(1 + s_1,\dots,1+s_k)$ 
is  proportional to $(1 + r_1,\dots,1+r_k)$. 
Therefore the semigroup $\mathcal{F}^{(\vs)}$ 
is simply a reparameterization
of the restriction of $\mathcal{F}^{(\vr)}$ to 
$\R^k$, and the claim follows
from  Lemma~\ref{lem: product}. \end{proof}

It is clear that the winning property of the set \equ{tripleint}
follows from a special case of the above proposition. 
The same  scheme of proof,  which seems to be much less involved than
that of \cite{PV-bad},  is applicable to multiple intersections of
sets of weighted \ba\ vectors. E.g.\ given $\vr\in\R^3$ with $r_1 +
r_2 + r_3 = 1$, 
equation \equ{defn s} 
can be used to define $s_{ij}$ for $i,j = 1,\dots,3$, $i\ne j$,
such that \eq{sum 1}{\sum_{i} s_{ij} = 1\text{ for  }j = 1,2,3\,,} and that
\eq{7int}{
\begin{aligned}
\BA(\vr) &\cap \BA(s_{13},s_{23},0)  \cap \BA(s_{12},0,s_{32}) \cap
\BA(0, s_{21},s_{31}) \\ &\cap \BA(1,0,0) 
\cap \BA(0,1,0)\cap \BA(0,0,1)
\end{aligned}}
is a winning set  for the MSG induced by $\mathcal{F}^{(\vr)}$,  and therefore is thick.
Take for example $\vr = (\frac12, \frac13, \frac16)$; our conclusion is that 
$$\begin{aligned}
\BA(\vr) &\cap \BA(\tfrac{10}{17}, \tfrac7{17},0)  \cap \BA(\tfrac9{16},0,\tfrac7{16}) \cap \BA(0, \tfrac2{3},\tfrac1{3}) \\ &\cap \BA(1,0,0)
\cap \BA(0,1,0)\cap \BA(0,0,1)
\end{aligned}$$
is thick. 
We remark that the assertion made in \cite[p.\ 32]{PV-bad}, namely that given $\vr$ as above, 
the set \equ{7int} is thick
for an {\it arbitrary\/} choice of $s_{ij}$ satisfying \equ{sum 1}, does not seem to follow from 
either our  methods of proof or those of \cite{PV-bad}. \commarginnew{changed} 

\subsection{Games and dynamics}\name{sec: dyn} 
The appearance of the semigroup $\mathcal{F}^{(\vr)}$ in our analysis of the set $\BA(\vr)$
can be naturally explained from the point of view of homogeneous dynamics.
Let $G=\SL_{n+1}(\R)$,  $\Gamma = \SL_{n+1}(\Z)$. The \hs\ 
$\ggm$ can be identified with the space of unimodular lattices in $\R^{n+1}$. To a vector $\x\in\R^n$
one associates a unipotent element $\tau(\x) = \left(
\begin{array}{ccccc}
I_n & \x \\ 0 & 1
\end{array}
\right)$ of $G$, which gives rise to a lattice $$\tau(\x) \Z^{n+1}  = \left\{\left(
\begin{array} {c} q \x - \p \\ q \end{array}
\right): q\in \Z,\ \p\in\Z^n\right\}\in\ggm\,.$$ Then, given $\vr$ as in 
\equ{defn r},  consider  the
one-parameter subgroup $ \{g^{(\vr)}_t\}$  of
$G$, where
\eq{def gtr}{g^{(\vr)}_t \df \diag(e^{r_1t}, \ldots,
e^{r_n t}, e^{-t})\,.
}
It was observed by Dani  \cite{Dani-div} for $\vr = \vn$ and by the first named author \cite{K-matrices}  for arbitrary $\vr$
that $\x\in\BA(\vr)$ if and only if the trajectory
$$\big\{g^{(\vr)}_t {\tau}(\x)\Z^{n+1} : t \ge 0\big\}$$ is bounded in $\ggm$.
Note that the $g^{(\vr)}_t$-action on $\ggm$ is partially hyperbolic, and  
it is straightforward to verify that the $\tau(\R^n)$-orbit foliation
is $g^{(\vr)}_t$-invariant, and that
 the action on the  foliation induced by the $g^{(\vr)}_t$-action on $\ggm$
is realized  by $\mathcal{F}^{(\vr)}$. Namely, one has
$$g^{(\vr)}_t  {\tau}(\x)y = g^{(\vr)}_t\tau\big({\Phi}^{(\vr)}_{-t}(\x)\big)y$$ for any $y \in\ggm$.


Dani used  Schmidt's result on the winning property of the set $\BA$ 
and the aforementioned 
correspondence 
 to prove that the set of points of $\ggm$ with bounded $g^{(\vn)}_t$-trajectories,
 where $\vn$ is as in \equ{def n}, is thick.
Later \cite{KM} this was established for arbitrary 
flows $(\ggm,g_t)$ 
`with no non-trivial quasiunipotent factors'. \commargin{changed this a little bit}
In fact the following was proved: denote by $H^+$
the $g_1$-expanding horospherical subgroup of $G$, that is,
$$H^+ = \{h\in G : g_{-t}hg_t \to e\text{ as }t\to\infty\}\,;$$ then for any $y\in\ggm$ 
the set
\eq{ehs}{
\big\{h\in H^+ : \{g_thy : t \ge 0\} \text{ is bounded in } \ggm\big\}
}
is thick.
The main result  of the present paper strengthens the above conclusion
in the case $G=\SL_{n+1}(\R)$,  $\Gamma = \SL_{n+1}(\Z)$ and $g_t = g^{(\vr)}_t $
as in \equ{def gtr}. Namely,
consider  the subgroup $H = \tau(\R^n)$
of $H^+$ (the latter for generic $\vr$ is isomorphic to the group of all upper-triangular unipotent matrices).
Then  for any $y\in\ggm$, the intersection of the set \equ{ehs} with an arbitrary coset $H h'$
of $H$ in $H^+$ is winning  for a certain MSG determined only by $\vr$ (hence is thick).
In particular, in view of Theorem \ref{thm: countable} and Proposition \ref{prop: affine commuting},
 this implies that for an arbitrary countable sequence of
points $y_k\in \ggm$, the intersection of all sets $
\big\{g\in G : \{g_t gy_k \} \text{ is bounded in } \ggm\big\}$
is thick.

We note that the proof in \cite{KM} is based on mixing of the $g_t$-action on $\ggm$, 
while to establish the aforementioned stronger winning property 
mixing does not seem to be enough,
and additional arithmetic considerations are necessary. 
In a recent work \cite{new}, for any 
flow $(G/\Gamma,g_t)$ with no nontrivial quasiunipotent factors we describe
a 
class of subgroups $H$ of the $g_1$-expanding horospherical subgroup of $G$ 
which are normalized by $g_t$ and have the property that for any $y\in\ggm$,
the set
\eq{bounded}{
\big\{h\in H : \{g_thy : t \ge 0\} \text{ is bounded in } \ggm\big\}
} 
is
winning for the MSG induced by  contractions $h\mapsto g_{-t}hg_t$. The argument is based on reduction theory for arithmetic
groups, that is, on an analysis of the structure of cusps of arithmetic \hs s.

Another result  obtained in \cite{new} is
that for $G$, $\Gamma$, $\{g_t\}$, $H$, $y$ as above and any $z\in\ggm$,  sets
\eq{nondense}{
\left\{h\in H : z\notin \overline{\{g_thy : t \ge 0\}} \right\}
} 
are also winning for the same MSG. Again this is a strengthening of results on the thickness
of
those sets existing in the literature, see \cite{K}. Combining the two statements above 
and using the intersection property of winning sets \equ{bounded} and \equ{nondense}, one
finds a way to construct orbits which are both bounded and stay away from a given 
countable subset of $\ggm$, which settles a conjecture made by Margulis in \cite{Ma}.

\subsection{Systems of linear forms}\name{sec: sys} 
A special case of the general theorem mentioned in the previous subsection 
is a generalization of the main
result of the present paper to the case of systems of linear forms. Namely,  let $m,n$ be positive integers,
denote by $\mr$ 
the space of $m\times n$ matrices with real entries (system
of $m$ linear forms in $n$ variables), and say that \amr\ is  
 {\sl $(\vr,\vs)$-\ba\/} if
$$ \inf_{\vp\in\Z^m,\,\vq\in\Z^n\nz} \max_i |Y_i\vq - p_i|^{1/r_i} \cdot
\max_j|q_j|^{1/s_j}    > 0\,,
$$
where $Y_i$, $i = 1,\dots,m$ are rows of $Y$ and $\vr\in\R^m$ and $\vs\in\R^n$ are such that
\eq{def rs}{r_i,s_j > 0\quad\text{and}\quad\sum_{i=1}^m r_i = 1 =
 \sum_{j=1}^n s_j\,.}
 (Here the components of vectors $\vr,\vs$ can  be thought of as weights assigned to
linear forms $Y_i$ and integers $q_j$ respectively.)
The correspondence described in the previous subsection extends to the matrix set-up, 
with  $G = \SL_{m+n}(\R)$ and $\Gamma = \SL_{m+n}(\Z)$
and 
\begin{equation*}
\label{eq: new defn g_t}
g^{(\vr,\vs)}_t =  \diag(e^{r_1t}, \ldots,
e^{r_{m} t}, e^{-s_1t}, \ldots,
e^{-s_{n}t} )
\end{equation*}
acting on $\ggm$.
This way  one can show  that the set $\BA(\vr,\vs)\subset \mr$ of $(\vr,\vs)$-\ba\ systems
is winning for the MSG induced by the semigroup of contractions
${\Phi}_t : (y_{ij}) \mapsto ( e^{- (r_i + s_j)t }y_{ij})$ of $\mr$ (a special case where all weights are
equal is a theorem of Schmidt \cite{Schmidt}). This generalizes  Theorem \ref{thm: main}
and strengthens \cite[Corollary 4.5]{di} where it was shown that $
\BA(\vr,\vs)
$ is thick
for any choice of $\vr,\vs$ as in \equ{def rs}.

\subsection{Playing games on other metric spaces}\name{sec: other} The paper  \cite{KTV}, where it
was first proved that the set of weighted badly approximable vectors in $\R^n$
has full \hd, contains a discussion of analogues of 
the sets  $\BA(\vr)$ over local fields other than $\R$. In \cite[\S\S5.3--5.5]{KTV} it is explained 
how to apply the methods of  \cite[\S\S2--4]{KTV} to studying weighted badly approximable vectors 
in vector spaces over $\C$ as well as over non-Archimedean fields\footnote{See also \cite{kr}
where  Schmidt's result on the winning property of the set of 
\ba\ systems of linear forms is extended to the field of formal power series.}. 
Similarly one can apply the methods of the present paper to replace Theorems 17--19
of  \cite{KTV} by stronger statements that the corresponding sets are winning sets of certain MSGs.
For that one needs to generalize the set-up of \S \ref{sec: contr} and consider
modified Schmidt games induced by contracting automorphisms of
arbitrary locally compact topological groups (not necessarily real Lie groups). 

Another theme of the papers  \cite{KTV} and  \cite{bad} is intersecting the set of badly
approximable vectors with some nice fractals in $\R^n$. For example  \cite[Theorem 11]{KTV}, 
slightly generalized in \cite[Theorem 8.4]{bad}, states the following: let  $\mu = \mu_1\times\dots\times \mu_d$, where each
$\mu_i$ is a
measure on $\R$ satisfying a power law (called `condition (A)' in \cite{KTV}); then $\dim\big(\BA(\vr)
\cap\supp\,\mu\big)  = \dim(\supp\,\mu)$. Following an approach developed recently by Fishman
\cite{Fishman}, it seems possible to strengthen this result; in particular, one can consider a modified
Schmidt game played on $E = \supp\,\mu$, with $\mu$ as above, and prove that the intersection of $E$ with $ \BA(\vr)$
is a winning set of this game.

\subsection{Schmidt's Conjecture}\name{sec: schmidt} Finally we would like to mention 
 a  question posed by W.\ Schmidt \cite{Schmidt:open} in 1982:
is it true that for $\vr\ne \vr'$, the intersection of $\BA(\vr)$ and $\BA(\vr')$ is nonempty? Schmidt
conjectured that the answer is affirmative in the special case $n = 2$, $\vr = (\frac13, \frac23)$
and $\vr' = (\frac23, \frac13)$, pointing out that disproving his conjecture would amount
to proving Littlewood's Conjecture (see  \cite{EKL} for its statement, history and recent
developments). 
Unfortunately, the results of the 
present paper do not give rise to any progress related to Schmidt's Conjecture. 
Indeed, each of the weight vectors $\vr$  comes with its own set of rules for the
corresponding modified Schmidt game, and there are no reasons to believe that 
winning sets of different games must have nonempty intersection. One can also observe
that $\BA(\frac23, \frac13) = {f}\big(\BA(\frac13, \frac23)\big)$ where ${f}$ is a reflection 
of $\R^2$ around the line $y = x$. This reflection however does not commute with $\mathcal{F}^{(1/3,2/3)}$, hence Theorem \ref{thm: precise} cannot be used to conclude\footnote{Recently a solution to the conjecture was announced by D.\ Badziahin, A.\ Pollington and S.\ Velani.} that ${f}\big(\BA(\frac13, \frac23)\big)$ is a winning set of the MSG induced by $\mathcal{F}^{(1/3,2/3)}$. \commargin{Should we mention the announcement of 
Pollington et al? I decided to add a footnote...}



\end{document}